\documentclass[a4paper,10pt]{amsart}

\usepackage[english]{babel}

\usepackage{amssymb,amsmath,latexsym,amsthm}

\usepackage[
            pdftitle={},
            pdfkeywords={Lattice counting, Quantum variance},
            pdfcreator={Pdflatex},
            pagebackref=true
            ]{hyperref}
\hypersetup{colorlinks=true}

\numberwithin{equation}{section}

\usepackage{xcolor,xspace}
\newcommand{\QVlink}{\hyperlink{targetQV}{\textcolor{blue!75!black}{QV}}\xspace}
\newcommand{\STXlink}{\hyperlink{targetSTX}{\textcolor{blue!75!black}{STX}}\xspace}

\theoremstyle{plain}
\newtheorem{theorem}{Theorem}[section]
\newtheorem{lemma}[theorem]{Lemma}
\newtheorem{proposition}[theorem]{Proposition}
\newtheorem{corollary}[theorem]{Corollary}

\theoremstyle{definition}

\newtheorem{remark}[theorem]{Remark}

\newtheorem*{QV}{Hypothesis QV}
\newtheorem*{STX}{Hypothesis STX}

\def \G {{\Gamma}}
\def \R {{\mathbb R}}
\def \H {{\mathbb H}}
\def \Q {{\mathbb Q}}
\def \GmodH {{\Gamma\backslash\H}}
\def \pslzi {{\hbox{PSL}_2({\mathbb Z}[i])} }
\def \pslz  {{\hbox{PSL}_2( {\mathbb Z})} }
\def \GmodHthree {{\Gamma\backslash\H^3}}

\newcommand{\norm}[1]{\left\lVert #1 \right\rVert}

\title{Local average in the hyperbolic sphere problem}

\author{Giacomo Cherubini}
\address{
   Istituto Nazionale di Alta Matematica ``Francesco Severi'',
   Research Unit Department of Mathematics ``Guido Castelnuovo'',
   Sapienza University of Rome,
   Piazzale Aldo Moro 5, I-00185, Rome,
   Italy
   }
\email{cherubini@altamatematica.it}

\author{Christos Katsivelos}
\address{University of Patras\\
Department of Mathematics\\
26504 Patras\\
Greece}
\email{up1112463@upatras.gr}

\date{\today}


\begin{document}

\begin{abstract}
We consider a local average in the hyperbolic lattice point counting problem
for the Picard group $\Gamma$ acting on the three-dimensional hyperbolic space.
Compared to the pointwise case, we improve the bounds on the remainder in the counting,
conditionally on a quantum variance estimate for Maass cusp forms attached to $\Gamma$.
We also use bounds on a spectral exponential sum over the Laplace eigenvalues for $\Gamma$,
which has been studied in the context of the prime geodesic theorem
and for which unconditional bounds are known.
\end{abstract}

\maketitle

\section{Introduction}\label{Introduction}

Let $\H^3$ be the three-dimensional hyperbolic space and let $z\in\H^3$.
The Picard group $\Gamma=\mathrm{PSL}_2(\mathbb{Z}[i])$ acts naturally
on $\H^3$ (see Section \ref{section3}) and therefore any $\gamma\in\Gamma$
moves $z$ to some point $\gamma z$, with the property
that in any given compact set there are only finitely many translates of $z$.
The problem of estimating the number of such translates in a
ball of growing radius is called hyperbolic sphere problem.

More specifically, let $X\geq 1$ and consider the function
\[
N(X,z) := \#\{\gamma\in\Gamma:\; \cosh d(z,\gamma z)\leq X\},
\]
where $d(z,w)$ denotes the hyperbolic distance.
By definition, $N(X,z)$ counts the number of lattice points $\gamma z$
inside a ball of radius $\cosh^{-1}X$. The presence of the hyperbolic cosine
is a common normalization in hyperbolic geometry, where distances grow exponentially,
and is chosen so that the main asymptotic for $N(X,z)$ becomes polynomial in $X$.

As one may expect, $N(X,z)$ is asymptotic to the volume of a hyperbolic ball of radius $\cosh^{-1}X$,
divided by the volume of a fundamental domain for $\Gamma$.
This can be proved with an explicit error term
\cite{elstrodt,laaksonen,laxphillips} in the form
\begin{equation}\label{1802:eq001}
N(X,z) = c_\Gamma X^2 + O(X^{3/2}),
\end{equation}
where (cf. ~\cite[p.312]{elstrodt}  and ~\cite[Eq.~(1.24)]{phirud})
\begin{equation*}
c_\Gamma=\frac{\pi}{2\mathrm{vol}(\GmodHthree)}=\frac{3\pi}{2\mathcal{C}}
\end{equation*}
and $\mathcal{C}\approx 0.915965$ is Catalan's constant.
Notice that the main term in \eqref{1802:eq001} is independent of $z$. The implied constant
in the error depends on $z$, but is uniformly bounded for $z$ in a compact set.

The exponent $3/2$ in \eqref{1802:eq001} has never been improved,
for no choice of the point~$z$, nor for groups other than the Picard group.
However, several types of averages have been studied \cite{hillparn,laaksonen,  phirud},
proving smaller exponents on average and hinting that the true
order of magnitude of the error in \eqref{1802:eq001} should be smaller than $O\left(X^{3/2}\right)$,
possibly of the order $O(X^{1+\epsilon})$ for every $\epsilon>0$.

The aim of this paper is to consider a local average of the problem
and measure how much we can reduce the exponent in the remainder
under the assumption of two hypotheses
that have been studied before in the literature:
\begin{itemize}
   \item a quantum variance bound for cusp forms attached to the group $\Gamma$;
   \item a bound on a spectral exponential sum over hyperbolic Laplace eigenvalues.
\end{itemize}
Moreover, we can compare our results with the case of the modular group $\pslz$
in two dimensions, where the same problem was studied by Petridis and
Risager \cite{petridisrisager} (and by Bir\'o \cite{biro} for any cofinite Fuchsian group).
In two dimensions, the first point
is in fact a theorem due to Luo and Sarnak \cite{luo sarnak},
and there has been a lot of progress on the corresponding spectral
exponential sum \cite{balkanova3, balkanova4,iwaniec 2, luo sarnak}.
In three dimensions, the analogous exponential sum has
been studied extensively in \cite{balkanova, BF, koyama},
but we are not aware of any non-trivial estimate on the quantum variance.

We require some more notation and terminology to state the hypotheses we need in rigorous terms,
and we postpone this to after Theorem \ref{ourmaintheorem},
see Hypothesis \QVlink 
in \S\ref{S1:QV} and Hypothesis \STXlink in \S\ref{S1:STX} respectively.
In addition, for a smooth compactly supported function $f$ on the Picard manifold
$\GmodHthree$, we define its average value by the integral
\begin{equation}\label{def:fbar}
\bar{f} = \frac{1}{\mathrm{vol}(\GmodHthree)} \int_{\GmodHthree} f(z)\,d\mu(z).
\end{equation}
Our main result can be stated as follows.

\begin{theorem}\label{ourmaintheorem}
Let $\G =\mathrm{PSL}_2(\mathbb{Z}[i])$ and
let~$f$ be a smooth compactly 
supported function on $\GmodH^3$,
with its average value $\bar{f}$
defined as in \eqref{def:fbar}.
Furthermore, assume:
\begin{itemize}
   \item Hypothesis \QVlink with the exponent $q\in[1,3]$;
   \item Hypothesis \STXlink with the exponents $(7/4+\theta,1/4)$, for some $\theta\in[0,1/4]$.
\end{itemize}
Then, for every $\epsilon>0$, we have
\begin{equation}\label{1802:eq002}
\int_{\GmodHthree} N(X,z)f(z)\,d\mu(z)=\frac{\pi}{2}\bar{f}X^2+O_{f,\epsilon,\theta,q}\left(X^{\frac{6-4\theta}{5-4\theta}+\epsilon}+X^{\frac{2q}{q+1}+\epsilon}\right).
\end{equation}
\end{theorem}

At this stage, even without knowing the exact formulation of Hypothesis \QVlink and Hypothesis \STXlink,
the reader can verify that
if $q<3$ (any $\theta\in[0,1/4]$), the remainder in \eqref{1802:eq002} is smaller than $O(X^{3/2})$,
improving \eqref{1802:eq001} in the local average.
See Remark~\ref{intro:rmk:variousbounds} for more comments on the exponents one obtains in the remainder,
depending on the values of $q$ and $\theta$.

\subsection{Quantum variance}\label{S1:QV}

Let us give the precise formulation of our assumptions.
The hyperbolic Laplace operator acts on the space
of square-summable functions on $\GmodHthree$
and admits a continuous spectrum as well as a discrete spectrum of eigenvalues
\[
0=\lambda_0<\lambda_1\leq \lambda_2\leq \cdots \leq \lambda_j\to\infty.
\]
Eigenvalues $\lambda_j\in [0,1]$ are called `small eigenvalues', but it is  known that for the
Picard manifold the only such eigenvalue is $\lambda_0=0$
and that in fact $\lambda_1>2\pi^2/3$, see \cite[Ch.~7, Proposition~6.2]{elstrodt}.
We can therefore write $\lambda_j=1+r_j^2$ with $r_0=i$ and $r_j\in [1,\infty)$ if $j\geq 1$.
For each $\lambda_j$ we denote by $u_j$ the associated eigenfunction and
if $\lambda_j$ occurs with multiplicity, it will be listed a corresponding number of times,
paired with distinct eigenfunctions.
One can construct the spectral measures
\[
d\mu_j(z) := |u_j(z)|^2d\mu(z).
\]
When $j\to\infty$,
these measures are expected to converge to the normalized uniform measure
on $\GmodHthree$, a prediction referred to as the
Quantum Unique Ergodicity Conjecture
and known to be true for the modular surface in two dimensions \cite{linderstrauss,soundararajan}.
In Theorem~\ref{ourmaintheorem} we assume an explicit quantitative version
of this result in square mean.

\hypertarget{targetQV}
{\begin{QV}
Let $\Gamma$ be the Picard group.
We say that Hypothesis QV holds with the exponent $q$
if for all $\epsilon>0$,
$T\gg 1$ and all smooth compactly supported functions $f$ on $\GmodHthree$, we have
\begin{equation}\label{2202:eq001}
\sum_{0<r_j\leq T}
\;\biggl|\int_{\GmodHthree} f(z) d\mu_j(z) - \bar{f}\biggr|^{2}
\ll_{f,\epsilon,q} T^{q+\epsilon},
\end{equation}
where $\bar{f}$ denotes the average value of $f$ as defined in \eqref{def:fbar}.
\end{QV}}

It is easy to check that the trivial bound in \eqref{2202:eq001} is $O_{f}(T^3)$ by the Weyl law.
As we mentioned earlier
\eqref{2202:eq001}
was proved for the modular group
by Luo and Sarnak \cite[Theorem 1.2]{luo sarnak}, with $q=1$
and with an explicit dependence on the test function~$f$.

Variance sums such as \eqref{2202:eq001} were introduced by Zelditch~\cite{zelditch1},
who proved a logarithmic saving over the trivial bound given by the Weyl law.
The result of Luo and Sarnak has been extended by Zhao \cite{Zhao}
and later by Sarnak and Zhao \cite{Sarnak Zhao} to non-holomorphic cusp forms.
We also mention that for any negatively curved, compact Riemannian manifold $\mathcal{M}$
of dimension at least~$2$,
it is conjectured (see e.g \cite[Equation (3)]{nelson})
that we should in fact have an asymptotic result of the form
\[
\sum_{\lambda_j\leq T^2} \left|\int_{\mathcal{M}}f(z)|u_j(z)|^2\,d\mu(z)-\frac{1}{\mathrm{vol}(\mathcal{M})}\int_{\mathcal{M}}f(z)\,d\mu(z)\right|^2=c_{\mathcal{M}}T+o(T),
\]
where $\{u_j\}$ is an orthonormal basis for the $\lambda_j$-eigenspace
of the Laplacian on $L^2(\mathcal{M})$ with respect to the measure $d\mu(z)$ and $f$ a test function.

\begin{remark}
In light of the above predictions and of the known results,
it has been conjectured 
\cite{feingold, nelson, zelditch2} that
Hypothesis \QVlink holds with $q=1$.
If we assume this in Theorem \ref{ourmaintheorem},
then the last term in \eqref{1802:eq002} is bounded by $O(X^{1+\epsilon})$
and is always dominated by the first term in the remainder, which is sensitive
to the strength of Hypothesis \STXlink.
\end{remark}

\subsection{A spectral exponential sum}\label{S1:STX}

Consider again the spectral parameters $r_j$ and
let $T,X>1$. The proof of Theorem \ref{ourmaintheorem}
involves the exponential sum
\[
S(T,X) = \sum_{0<r_j\leq T} X^{ir_j} .
\]
Bounding in absolute value each summand
and using the Weyl law (see \eqref{gqweyllaw}),
one gets $S(T,X)=O(T^3)$, uniformly in $X$.
However, it is reasonable to expect some cancellation in the sum.

\hypertarget{targetSTX}
{\begin{STX}
Let $\Gamma$ be the Picard group. We say that Hypothesis STX
holds with the exponents $(\alpha,\beta)$ if, for all $\epsilon>0$
and all $T,X\gg 1$, we have
\[
S(T,X) \ll_{\epsilon,\alpha,\beta} T^{\alpha+\epsilon} X^{\beta+\epsilon} + T^2.
\]
\end{STX}}

By the above observation, Hypothesis STX holds with the exponents $(3,0)$,
even without the additional $\epsilon$-powers nor the secondary term $T^2$.
Such a term is included to account for the number of spectral parameters
$r_j$ in a unit interval at height~$T$. In very short windows
the eigenvalues could cluster together and so it may be difficult
to detect cancellation at that level.

The sum $S(T,X)$ appears naturally in the spectral analysis of the prime geodesic
theorem and has been studied extensively in that context,
both in two and three dimensions.

For the Picard group, it was proved in \cite{balkanova}
that Hypothesis \STXlink holds with the exponents $(2,1/4)$.
Balkanova and Frolenkov \cite{BF} showed that the pair
$(7/4+\theta,1/4)$ is allowed, where $\theta$ is any subconvexity exponent
for quadratic Dirichlet $L$-functions over the Gaussian integers.
By standard convexity estimates one can take $\theta=1/4$, recovering
the exponents $(2,1/4)$. Nelson \cite{nelson2} announced that $\theta=1/6$ holds,
which would match Weyl's subconvexity exponent for Dirichlet
$L$-functions over the integers \cite{CI,PY1,PY2}.
Qi's work \cite{Qi} (see \cite[Theorem 2]{Qi} and combine it with the comment after
\cite[Corollary 3.4]{balkanova}) gives Hypothesis \STXlink with the exponents $(15/8,1/4)$.

\begin{remark}\label{intro:rmk:variousbounds}
In the remainder in Theorem \ref{ourmaintheorem} the last term dominates as long as $q\geq (3-2\theta)/(2-2\theta)$.
When $q$ is below this value, Hypothesis \STXlink comes into play.
If we use the pair $(2,1/4)$, corresponding to the convexity exponent $\theta=1/4$,
and we assume Hypothesis \QVlink with any $q\leq 5/3$,
then the remainder in \eqref{1802:eq002} is bounded by $O(X^{5/4+\epsilon})$.
Instead, if we use the the strongest estimate on $S(T,X)$, corresponding to $\theta=0$,
and we further assume Hypothesis \QVlink with $q=1$,
then we see that the remainder in \eqref{1802:eq002} is bounded by $O(X^{6/5+\epsilon})$.
Notice that actually any $q\leq 3/2$ suffices in this case.
\end{remark}

\begin{remark}
It has been conjectured \cite[Conjecture 5.1]{kaneko 2}
that Hypothesis \STXlink holds with the exponents $(2,0)$.
If we assume this and Hypothesis \QVlink with $q=1$, then our proof will give
the optimal error term in the local average,
in analogy with the two-dimensional case,
see~\cite[Remark 6.4]{petridisrisager}.
\end{remark}

\begin{remark}
Lower bounds on the remainder in the $n$-dimensional hyperbolic lattice point counting problem
are known by the work of Phillips and Rudnick:
they proved that for any cocompact or congruence group the remainder satisfies
\begin{equation}\label{2507:eq001}
\Omega(X^{\frac{n-1}{2}}(\log\log X)^{\frac{n-1}{2n}-\epsilon}),
\end{equation}
for any $\epsilon>0$, see \cite[Eq.~(4.24)]{phirud}.
When $n=2$, a lower bound of the same quality has been proved by Petridis and Risager in the local average
for the modular group, see \cite[Theorem~1.2]{petridisrisager}.
When $n=3$, we expect that adapting the methods of \cite{petridisrisager,phirud}
one should be able to match again the pointwise lower bound
\eqref{2507:eq001} also in the local average for the Picard group.
\end{remark}

\subsection{Outline of the proof}

Let us give a sketch of the proof of Theorem \ref{ourmaintheorem},
without being too rigorous but rather trying to point out where Hypothesis \QVlink and Hypothesis \STXlink are used.
The spectral expansion of the counting function $N(X,z)$ gives, very roughly,
\[
N(X,z) = c_{\G}X^2 + \sum_{0<r_j\leq T} \frac{X^{1+ir_j}}{r_j^2}|u_j(z)|^2 + O\left(\frac{X^{2}}{T}\right),
\]
where $T>1$ is a parameter
(we are making several simplifications here: the sum is in fact infinite with a smoothing; we omit the tail since it will contribute as much as the part with $r_j$ up to $T$; the coefficients behave like $r_j^{-2}$ only asymptotically; and one should take the real part of the sum; the error comes from the smoothing of the main term).

We wish to remove the eigenfunctions $|u_j(z)|^2$ and we do this by
means of a local average: integrating against a smooth compactly supported
function $f$ as in Theorem \ref{ourmaintheorem} gives
\[
\int_{\GmodHthree} N(X,z)f(z)d\mu(z) = \frac{\pi}{2}\bar{f} X^2 + \sum_{0<r_j\leq T} \frac{X^{1+ir_j}}{r_j^2}\int_{\GmodHthree} f(z)d\mu_j(z) + O\left(\frac{X^{2}}{T}\right).
\]
We add and subtract the quantity
\[
\bar{f}\sum_{0<r_j\leq T}\frac{X^{1+ir_j}}{r_j^2}
\]
and apply Cauchy--Schwarz, obtaining
\[
\begin{split}
    &\int_{\GmodHthree}N(X,z)f(z)\,d\mu(z)-\frac{\pi}{2}\bar{f}X^2-\bar{f}\sum_{0<r_j\leq T}\frac{X^{1+ir_j}}{r_j^2}
    \ll \frac{X^2}{T}
    \\
    &+
    X\Biggl(\sum_{0<r_j\leq T}\frac{1}{r_j^3}\Biggr)^{\!\!1/2} \!
    \Biggl(\sum_{0<r_j\leq T}\frac{1}{r_j}\biggl|\int_{\GmodHthree}f(z)|u_j(z)|^2\,d\mu(z)-\int_{\GmodHthree}f(z)\,d\mu(z)\biggr|^2\Biggr)^{\!\!1/2}.
\end{split}
\]
The sum of $r_j^{-3}$ gives $O(\log T)$ and on the second sum we apply
the quantum variance bound from Hypothesis \QVlink, which leads to
\[
\int_{\GmodHthree}N(X,z)f(z)\,d\mu(z)-\frac{\pi}{2} \bar{f}X^2-X\bar{f}\sum_{0<r_j\leq T}\frac{X^{ir_j}}{r_j^2}
\ll \frac{X^2}{T}+XT^{\frac{q-1}{2}+\epsilon}.
\]
Next on the oscillating sum with $X^{ir_j}$
we use Hypothesis \STXlink with the exponents $(3,0)$ and $(7/4+\theta,1/4)$.
By interpolation (ignore the secondary term $T^2$), they give the pair
$(2,(1-4\theta)/(5-4\theta))$, which in turn gives
\[
\int_{\GmodHthree}N(X,z)f(z)\,d\mu(z)-\frac{\pi}{2}\bar{f}X^2
\ll
\frac{X^2}{T} + X^{\frac{6-4\theta}{5-4\theta}+\epsilon}T^{\epsilon} + XT^{\frac{q-1}{2}+\epsilon}.
\]
Finally, we pick $T=X^{\frac{2}{q+1}}$ in order to balance the first and the last term
on the right, which yields
\[
\int_{\GmodHthree}N(X,z)f(z)\,d\mu(z) = \frac{\pi}{2} \bar{f} X^2
+
O\bigl( X^{\frac{6-4\theta}{5-4\theta}+\epsilon}+X^{\frac{2q}{q+1}+\epsilon}\bigr),
\]
as stated in Theorem \ref{ourmaintheorem}.

\subsection{Outline of the paper}

The paper is organized as follows: in Section \ref{section3}
we explain the geometric definition of the problem
and compute the Selberg transform of the indicator function of a ball.
In Section \ref{S4} we perform the spectral expansion of
the (smooth) automorphic kernel associated to $N(X,z)$
and apply Hypothesis \QVlink to remove the eigenfunctions.
In Section \ref{S5} we apply Hypothesis \STXlink
and remove the smoothing, concluding the proof of Theorem \ref{ourmaintheorem}.

\subsection*{Acknowledgements}
We thank Dimitrios Chatzakos for helpful discussions.
We thank the reviewer for their comments on the paper.
G.C.~is member of the INdAM group GNSAGA.
C.~K.~was supported by the Hellenic Foundation for Research and Innovation (H.F.R.I.) under the “3rd Call for H.F.R.I. Research Projects to support Faculty Members \& Researchers” (Project Number: 25622).

\section{Initial Settings} \label{section3}

We introduce here terminology and notation that will be used throughout the paper.
Most of the material in this section is standard and is taken from \cite{elstrodt}.
The group $\Gamma$ is fixed once and for all to be the Picard group $\pslzi$.

\subsection{Geometric definitions}
The hyperbolic three-dimensional space $\H^3$ is described in the upper half-space model
as the set of points $z=(x_1,x_2,y)$, with $x_1,x_2\in\mathbb{R}$ and $y>0$.
Equivalently, a point $z\in\H^3$ can be written as a quaternion $z=x+jy$, where $x\in\mathbb{C}$ and $j^2=-1$.
If we do so, then the action of the Picard group
on $\H^3$ is given by generalized linear fractional
transformations: an element
$\gamma=\left(\begin{smallmatrix}
    a & b\\
    c & d
\end{smallmatrix}\right)\in\Gamma$
acts on a point $z\in\H^3$ by
\[\gamma z=(az+b)(cz+d)^{-1},\]
where inverse and multiplication are taken in the skew field of quaternions.
The metric on $\H^3$ is given by the line element $ds^2=y^{-2}(dx_1^2+dx_2^2+dy^2)$.
This gives rise to a hyperbolic volume measure $\mu$
whose corresponding volume element is $d\mu=y^{-3}dx_1dx_2dy$.
An easy calculation shows that 
the volume of a ball of radius~$R$ with respect to~$\mu$ equals
\begin{equation*}
\mu(B_R)=\pi\left(\sinh(2R)-2R\right).
\end{equation*}
Notice that, when $R\to 0$, the
above volume is asymptotic to $4\pi R^3/3$, that is, the volume of the standard ball in $\mathbb{R}^3$.

Let us denote by $d$ the distance induced by the above metric.
It is sometimes handy to express $d(z,w)$ in terms of the function
\begin{equation*}
    \delta(z,w)=\frac{|x_z-x_w|^2+y_z^2+y_w^2}{2y_zy_w},
\end{equation*}
where $|x_z-x_w|$ is the usual absolute value in $\mathbb{C}$. By \cite[Ch.1 Proposition 1.6]{elstrodt}, we have
\begin{equation}\label{invariantmetricrelation}
    \cosh d(z,w)=\delta(z,w).
\end{equation}
Therefore we can pass without much effort from $d$ to $\delta$ and vice versa.

\subsection{Selberg transform of the indicator function of a ball}

Let $R\geq 0$. We consider the indicator function of a ball in $\H^3$, namely
\begin{equation}\label{kernelgiacomo}
k_R(\delta(z, w)):=
\mathbf{1}_{[1,\cosh R]}(\delta(z,w))
=
\begin{cases}
1 & \text{if $\delta(z,w)\leq \cosh R$}, \\
0 & \text{otherwise}.
\end{cases}
\end{equation}
By \eqref{invariantmetricrelation},
the above condition on $\delta$ corresponds to $d(z,w)\leq R$.
Functions $k$ that depend only on the hyperbolic distance between points in $\H^3$,
such as $k_R$, are called point-pair invariants. For these, we define
the Selberg transform by the formula \cite[Ch.3, Eq. (5.3)]{elstrodt}
\begin{equation}\label{giacomospectralfunction}
    h(r)=\frac{4\pi}{r}\int_{0}^{\infty}k(\cosh u)\sin (ru)\sinh u\,du.
\end{equation}
When $r=0$, the above is understood as the limit $r\to 0$.
Note that $h$ is even and that if $k$ is compactly supported and bounded (as is $k_R$),
then $h$ is well defined for all complex values of $r$.
The Selberg transform of $k_R$ is
\begin{equation}\label{ourspectral}
    h_R(r)=\frac{2\pi\sinh \left(R(1+ir)\right)}{ir(1+ir)}-\frac{2\pi\sinh \left(R(1-ir)\right)}{ir(1-ir)},
\end{equation}
for all $r\in\mathbb{C}$ with the exception of $r=0,\pm i$.
At these points we have
\begin{equation}\label{specialvalues}
\begin{gathered}
h_R(0) = 4\pi(R\cosh R-\sinh R),\\
h_R(\pm i) = \pi(\sinh (2R)-2R).
\end{gathered}
\end{equation}

Next, we construct approximations to $k_R$ from above and below, in such a way 
that the corresponding Selberg transforms will have better decay properties than~$h_R$.
Let $k,k'$ be two point-pair invariant functions. The hyperbolic convolution of $k$ and $k'$ is defined as 
\begin{equation*}
   ( k\ast k')(\delta(z,w))=\int_{\H^3}k(\delta(z,u))k'(\delta(u,w))\,d\mu(u).
\end{equation*}

As showed in \cite[4. p.323]{cham2} in the two-dimensional case
(we omit the proof in $\H^3$),
the Selberg transform turns convolution into pointwise product.
Using this observation, for $0<\eta<1$ and $R\gg 1$ we define 
\begin{equation*}\label{kernelpm}
    k_{\pm}(\delta(z,w)):=\frac{\left(k_{R\pm\eta}\ast k_{\eta}\right)(\delta(z,w))}{\mu(B_\eta)},
\end{equation*}
so that the corresponding Selberg transform is 
\begin{equation}\label{gc spectral}
    h_{\pm}(r)=\frac{h_{R\pm\eta}(r)h_{\eta}(r)}{\mu(B_\eta)}.
\end{equation}
Note that we have 
\begin{equation}\label{gckernelpm}
k_{\pm}(\delta(z, w))=
\begin{cases}
1 & \text{if } d(z,w)\leq R\pm\eta-\eta,\\
0 & \text{if } d(z,w)\geq R\pm\eta+\eta
\end{cases}
\end{equation}
and $0\leq k_{\pm}\leq 1$ in the remaining range.
Therefore we obtain the inequality 
\begin{equation}\label{giacomoinequality}
    k_{-}\leq k_R\leq k_{+}.
\end{equation}
Regarding $h_{\pm}$ we record the following.

\begin{lemma}\label{gcl1}
    Let $0<\eta<1,\, R\gg1$ and let $h_\pm$ be as in \eqref{gc spectral}.
    For any $L>0$, the function $h_\pm$ is even and holomorphic in the strip $|\Im(r)|<L$.
    In such a strip, we have the estimate $h_{\pm}(r)\ll_{L,R,\eta}(1+|r|)^{-4}$.
    When $r\in\mathbb{R},$ we have
    \begin{equation}\label{hpmestimate}
        h_{\pm}(r)\ll\frac{Re^R}{(1+|r|)^2(1+\eta|r|)^2},
    \end{equation}
    where the implied constant is absolute.
    When $r\geq 1$, we can write
    \begin{equation}\label{0202:eq001}
    h_{\pm}(r) = A(R,r,\eta)e^{ir(R\pm\eta)} + B(R,r,\eta)e^{-ir(R\pm \eta )},
    \end{equation}
    where
    \begin{equation}\label{0901:eq010}
    \begin{gathered}
    A(R,r,\eta),B(R,r,\eta)\ll e^{R} r^{-2}\min\{1,(\eta r)^{-2}\},\\
    \partial_r A(R,r,\eta),\partial_r B(R,r,\eta) \ll e^{R} r^{-3} \min\{1,(\eta r)^{-1}\}.
    \end{gathered}
    \end{equation}
    The implied constants in \eqref{0901:eq010} are uniformly bounded in $R,r\geq 1$ and $\eta\in(0,1)$.
\end{lemma}

\begin{proof}
The holomorphicity of $h_\rho(r)$ in any strip $|\Im r|<L$ follows directly from
\eqref{ourspectral} and \eqref{specialvalues}.
Let $r$ be in such a strip, 
with $|r|\gg 1$.
Applying the inequality
$|\sinh (\rho(1+ir))|\leq e^{(L+1)\rho}$
in \eqref{ourspectral}
and recalling the definition of $h_\pm$ in \eqref{gc spectral},
we deduce the desired bound $h_{\pm}(r)\ll_{L,R,\eta}(1+|r|)^{-4}$.
Next, let us discuss \eqref{hpmestimate}.
For any $\rho>0$ and $r\in\R$ we have that 
\begin{equation}\label{0901:eq001}
|h_\rho(r)|\leq h_\rho(0) \leq \mu(B_\rho).
\end{equation}
Indeed, using \eqref{kernelgiacomo}--\eqref{giacomospectralfunction} and by the inequality $|\sin(ru)|\leq |ru|$, valid for all $u,r\in\R$, we have that 
\begin{align*}
|h_\rho(r)|
&=
\frac{4\pi}{|r|} \int_0^\rho |\sin(ru)| \sinh(u) du
\\
&\leq 4\pi \int_0^\rho u \sinh(u) du
\;=h_\rho(0)\leq 4\pi \int_0^\rho \sinh^2(u) du = h_\rho(\pm i) = \mu(B_\rho).
\end{align*}
We claim that the following estimate holds for all $r\in\mathbb{R}$, $R\geq 1$ and $\eta\in(0,1)$:
\begin{equation}\label{july:eq001}
h_\pm(r) \ll
Re^R\min\left\{1,\frac{1}{r^2},\frac{1}{r^4\eta^2}\right\},
\end{equation}
where $r=0$ is understood as the limit $r\to 0$ and from which the stated bound \eqref{hpmestimate}
easily follows.
To prove \eqref{july:eq001}, we apply \eqref{0901:eq001} to $h_\eta(r)$ obtaining $|h_\eta(r)|\leq \mu(B_\eta)$ and also to $h_{R\pm\eta}(r)$ obtaining the estimate $|h_{R\pm\eta}(r)|\leq h_{R\pm\eta}(0)\ll R e^R$, uniformly in $\eta$. Combining the above in the definition \eqref{gc spectral} of $h_\pm(r)$ yields
\[
h_\pm(r) \ll Re^{R},
\]
which gives the first term in the minimum.
As for the second and third bound
in \eqref{july:eq001},
it suffices to prove them when $|r|\geq 1$, for otherwise the bound we just proved implies the result. 
When $|r|\geq 1$, we apply again \eqref{0901:eq001} to $h_\eta(r)$, while for $h_{R\pm\eta}(r)$ we can rewrite \eqref{ourspectral} as 
\begin{equation}\label{0901:eq002}
h_R(r) = \frac{4\pi}{r(1+r^2)}\bigl(\cosh(R)\sin(Rr)    -r\sinh(R) \cos(Rr) \bigr).
\end{equation}
Bounding $\cosh(R)\ll e^R$, $\sin(R r)\ll R|r|$, $\sinh(R)\ll e^R$ and $\cos(Rr)\ll 1$, we obtain
\[
|h_\pm(r)| \leq |h_{R\pm\eta}(r)| \ll \frac{Re^R}{r^2},
\]
which proves the second bound in \eqref{july:eq001}.
Finally, to prove the last bound in \eqref{july:eq001},
we apply \eqref{0901:eq002} to $h_\eta(r)$ and bound
$\sinh(\eta)\ll\eta$, $\sin(\eta r)\ll \eta |r|$, $\cos(\eta r)\ll 1$ and lastly $\cos(\eta)\ll 1$,
getting
\[
h_\eta(r) \ll \frac{\eta}{r^2},
\]
which yields, recalling $\mu(B_\eta)\asymp \eta^3$,
\[
h_\pm(r) \ll \frac{Re^R}{r^2} \times \frac{\eta}{r^2} \times \frac{1}{\eta^3} \ll \frac{Re^R}{r^4\eta^2}
\]
as claimed.
Finally, \eqref{0202:eq001} and \eqref{0901:eq010} follow from \eqref{ourspectral}
and similar estimates.
\end{proof}

\begin{remark}
    As the proof of Lemma \ref{gcl1} shows, the bound in \eqref{hpmestimate} is not sharp,
    because we apply
    the uniform bound $Re^{R}$ 
    all the time.
    This choice comes at the cost of a polynomial loss of magnitude $R$, which does not affect our result,
    but helps for the brevity of our proofs.
\end{remark}

\subsection{Approximation in the circle problem.}

The counting function $N(X,z)$ defined in the introduction
gives the number of lattice points in a ball of radius $R=\cosh^{-1}X$.
As described above in this section, many geometric quantities
are naturally expressed in terms of the variable $R$ rather than
the variable $X$. Because of this, sometimes we find convenient
to pass from one normalization to the other.
By a small abuse of notation, we write $N(R,z)$ in place of $N(X,z)$,
so that we can state our results in terms of $R$. Recalling that $k_R$
is the indicator function of a ball of radius $R$ (see \eqref{kernelgiacomo}),
we can thus write
\begin{equation*}
    N(R,z) = \# \{ \gamma \in \G : d(z,\gamma z) \leq R \} = \sum_{\gamma\in \Gamma} k_R(\delta(z,\gamma z)).
\end{equation*}
When $R\to\infty$, from \eqref{1802:eq001} we get that
\begin{equation}\label{main and error in R}
N(R,z)\sim c_\Gamma e^{2R}, \quad c_\Gamma=\frac{3\pi}{2\mathcal{C}} .
\end{equation}
Let $0<\eta<1,\,R\gg 1$ and let $k_{\pm}$ be the functions
defined in \eqref{gckernelpm}.
Recalling that they approximate $k_R$ from above and below (see~\eqref{giacomoinequality}), if we define
\begin{equation}\label{gcmodifiedautokernel}
    N^{\pm}(R,z):=\sum_{\gamma\in\G}k_{\pm}(\delta(z,\gamma z)),
\end{equation}
then  we have the inequality
\[
N^{-}(R,z)\leq N(R,z)\leq N^{+}(R,z).
\]
Subtracting the main term given in \eqref{main and error in R}
and integrating against $f$ we deduce, for any $\eta\in(0,1)$, the inequality
\begin{equation*}
     \biggl|
     \int_{\GmodHthree} (N(R,z)-c_\Gamma e^{2R})f(z)\,d\mu(z)
     \biggr|
     \leq \max_{\pm}
     \biggl|
     \int_{\GmodHthree} (N^{\pm}(R,z)-c_\Gamma e^{2R})f(z)\,d\mu(z)
     \biggr|.
\end{equation*}
Therefore, in order to prove Theorem \ref{ourmaintheorem}, it suffices to show the following.
\begin{proposition}\label{smoothedkernel}
    Let $\epsilon>0$, $R\gg1$, and for any $\eta\in(0,1)$ let $N^{\pm}(R,z)$ be as in \eqref{gcmodifiedautokernel}.
    Let $f$ be a smooth compactly supported function as in Theorem~\ref{ourmaintheorem}.
    Assume Hypothesis \QVlink holds with the exponent $q\in[1,3]$
    and Hypothesis \STXlink holds with the exponents $(7/4+\theta,1/4)$, $\theta\in[0,1/4]$.
    Then, for a suitable choice of $\eta=\eta(R,\theta,q)$, we have
    \begin{equation*}
        \int_{\GmodHthree}N^{\pm}(R,z)f(z)\,d\mu(z)=\frac{\pi}{2}\bar{f}e^{2R}+O_{f,\epsilon,\theta,q}\left(e^{R\left(\frac{6-4\theta}{5-4\theta}+\epsilon\right)}+e^{R\left(\frac{2q}{q+1}+\epsilon\right)}\right)
    \end{equation*}
    as $R\to\infty$.
    The constant $\bar{f}$ is the average value of $f$ as defined in \eqref{def:fbar}.
\end{proposition}

\section{Spectral expansion and quantum variance bound}\label{S4}

The spectral theory of automorphic forms allows us to expand
geometric quantities such as $N^\pm(R,z)$
in terms of the eigenfunctions of the hyperbolic Laplacian.
Let $k:[1,\infty)\to\mathbb{C}$ and consider the automorphic kernel
\begin{equation*}
    K(z,w)=\sum_{\gamma\in\G}k(\delta(z,w)),
\end{equation*}
where the function $k$ is assumed to decay rapidly enough to ensure convergence.
Let $\{u_j\}$ be an orthonormal system of Maass cusp forms for $\G$.
We write the corresponding eigenvalues as $\lambda_j=1+r_j^2$
and $r_j>0$ when $\lambda_j>1$. For the Picard group it is known
that there are no small eigenvalues other than $\lambda_0=0$ \cite[Ch.~7, Proposition~6.2]{elstrodt}.
There is one cusp and therefore one Eisenstein series $E(z,s)$,
which we assume to be normalised so that the critical line is at $\Re(s)=1$ (see \cite[Ch.~3 Eq.~(2.5)]{elstrodt}
for the definition of Eisenstein series).
Then we have the following:
\begin{lemma}\label{gcl2}
    Assume that $k:[1,\infty)\to\mathbb{C}$ is such that its Selberg transform $h$ (defined in \eqref{giacomospectralfunction}) is even, holomorphic in a strip $|\Im(r)|<1+\xi$ for some $\xi>0$ and satisfies $h(r)\ll (1+|r|)^{-3-\xi}$ in the strip. Let $f$ be a smooth, compactly supported function on $\GmodH^3$. Then, we have
    \begin{equation*}
    \begin{split}
        \int_{\GmodH^3}K(z,z)f(z)\,d\mu(z) ={}& \sum_{r_j}h(r_j)\int_{\GmodH^3}f(z)|u_j(z)|^2\,d\mu(z)
        \\
        &+
        \frac{1}{4\pi}\int_{\GmodH^3}\int_{\mathbb{R}}|E(z,1+ir)|^2h(r)f(z)\,dr\,d\mu(z),
    \end{split}
    \end{equation*}
    where the right-hand side converges absolutely.
\end{lemma}

\begin{proof}
    The spectral expansion of $K(z,z)$ \cite[Ch.6 Theorem 4.1]{elstrodt} gives the identity for a fixed $z$. Multiplying against $f$ and integrating over $\GmodHthree$ gives the lemma. We remark that the statement in \cite[Ch.6 Theorem 4.1]{elstrodt} is limited to test functions $k$ in the Schwartz class, but the requirements can be relaxed to the conditions on $h$ given above (cf.\cite[Theorem 7.4]{iwaniec}). The constant $(4\pi)^{-1}$ in the second line may depend on the normalization of the Eisenstein series (see \cite[pp. 98--100]{elstrodt}). For the Picard group, the value of the constant appears e.g.~in \cite[\S4.4.4]{laaksonen}.
\end{proof}
\par
The assumption $h(r)\ll|r|^{-3-\xi}$ suffices to counterbalance the eigenfunctions $|u_j(z)|^2$, since by the local Weyl law we know that
\begin{equation}\label{gqlocalweylS4}
    \sum_{r_j\leq T}|u_j(z)|^2\ll T^3.
\end{equation}
The implied constant is uniform for $z$ in a compact set, so the same type of bound as in \eqref{gqlocalweylS4} holds when we integrate against a compactly supported function, which shows that the series over the discrete spectrum in Lemma \ref{gcl2} converges absolutely. In fact, \eqref{gqlocalweylS4} implies
that $|u_j(z)|^2$ is $O(1)$ on average, since Weyl's law \cite[Ch.8 Theorem 9.1]{elstrodt} gives the asymptotic
\begin{equation}\label{gqweyllaw}
    \#\{r_j\leq T\}\sim\frac{T^3}{3\pi^4}.
\end{equation}

Performing the integration against $f$ and using Hypothesis \QVlink
will allow us to replace the cusp forms identically by
a constant and obtain a sum featuring the Selberg transform only.
Indeed, the cuspidal contribution in Lemma \ref{gcl2} features the integral of $f$ against the spectral measures
\begin{equation*}
d\mu_j(z):=|u_j(z)|^2d\mu(z).
\end{equation*}
As we mentioned before, the convergence of these measures to the uniform measure $d\mu(z)$,
in the limit $j\to\infty$, is known as the quantum unique ergodicity conjecture.
The assumption in Hypothesis \QVlink provides explicit estimates
for the speed of convergence, on average over the spectral parameters $r_j$.

\begin{lemma}\label{gcl3}
  Let $0<\eta<1,\, R\gg1$, and let $h_\pm$ be as in \eqref{gc spectral}
  and let $f$ be a smooth, compactly supported function on $\GmodHthree$.
  Assume Hypothesis \QVlink holds with the exponent $q\in[1,3]$.
  Then for every $\epsilon\in(0,1)$ we have
  \[
      \sum_{r_j}h_\pm(r_j)\int_{\GmodHthree}f(z)\,d\mu_j(z)
      =
      \bar{f}\, h_{\pm}(i)
      +
      \bar{f}\sum_{r_j>0}h_\pm(r_j)+O_{f,\epsilon,q}\left(Re^{R}\eta^{-\frac{q-1}{2}-\epsilon}\right).
  \]
\end{lemma}
\begin{proof}
    The contribution from the zero eigenvalue (i.e.~$r_j=i$) appears unaltered on both sides,
    so we can assume $\lambda\neq 0$.
    There are no small eigenvalues for the Picard group and in fact $r_j>1$ \cite[Ch.7 Proposition 6.2]{elstrodt}.
    In a dyadic window $r_j\in(T,2T]$ we can bound, using Cauchy-Schwarz inequality, Weyl's law and Hypothesis \QVlink,
    \[
        \mathcal{D}_T:=\!\!\sum_{T<r_j\leq 2T} \!\! h_\pm(r_j)\biggl(\int_{\GmodHthree}f(z)\,d\mu_j(z)-\bar{f}\biggr)\ll_{f,\epsilon,q} T^{\frac{3+q}{2}+\epsilon}\!\max_{T<r\leq 2T}|h_\pm(r)|.
    \]
    By Lemma \ref{gcl1}, we deduce
    \begin{equation*}
        \mathcal{D}_T\ll_{f,\epsilon,q} Re^RT^{\frac{q-1}{2}+\epsilon}\min(1,(\eta T)^{-2})\ll_{f,\epsilon,q} Re^R T^{-\epsilon}\eta^{-\frac{q-1}{2}-\epsilon}.
    \end{equation*}
    The result follows from a dyadic decomposition of the sum with $T=2^n$, $n\geq 0$.
\end{proof}

Concerning the contribution of
the continuous spectrum, we can show that it is essentially negligible
for the purpose of proving Theorem \ref{ourmaintheorem}.
\begin{lemma}\label{S4:lemma:continuous}
    Let $\eta\in(0,1)$ and $R\gg 1$.
    Let $k_\pm$ be as in \eqref{kernelpm} and let $h_\pm$ be its
    Selberg transform as defined in \eqref{gc spectral}.
    Let $E(z,s)$ denote the Eisenstein series for the Picard group.
    Let $f$ be a smooth compactly supported function on $\GmodH^3$.
    Then
    \begin{equation}\label{gqcontinuous}
        \int_{\GmodH^3}\int_\R|E(z,1+ir)|^2h_\pm(r)f(z)\,dr\,d\mu(z)\ll_f R\,e^{R}.
    \end{equation}
    The implied constant does not depend on $\eta$.
\end{lemma}
\begin{proof}
We follow \cite[Lemma 6.1]{petridisrisager} and start by considering the integral over $z$.
The support of $f$ is contained in a suitable compact part
of the fundamental domain for $\Gamma$, where the cusp has been cut off
at some height $a$ depending on $f$, so that we can bound
\begin{equation}\label{gqtruncatedint}
\begin{split}
\int_{\GmodHthree} f(z)|E(z,1+ir)|^2\,d\mu(z)\ll_f & \int_{y(z)\leq a} |E(z,1+ir)|^2\,d\mu(z)
\\
&\ll_f
\norm{E^{a}(z,1+ir)}^2,
\end{split}
\end{equation}
where $E^{a}(z,1+ir)$ is the truncated Eisenstein series
as defined in \cite[Ch.~6, Equations (3.15)--(3.17)]{elstrodt}.
The Maass--Selberg relations imply that for $\Re(s)\neq 1$,
\[
     \norm{E^{a}(z,s)}^2=C_\infty \left(\frac{a^{s+\overline{s}}-\phi(s)\overline{\phi(s)}a^{-s-\overline{s}}}{s+\overline{s}}+\frac{\overline{\phi(s)}a^{s-\overline{s}}-\phi(s)a^{-s+\overline{s}}}{s-\overline{s}}
     \right),
\]
where $C_\infty$ depends only on the group $\G$ and $\phi(s)$ is the scattering matrix
(see \cite[Ch.~6, Equations~(3.18)--(3.19)]{elstrodt} for the Maass--Selberg relations
and \cite[Ch.~6, Equation (1.4), p.~232]{elstrodt} for the definition of $\phi(s)$).
Taking the limit as $\Re(s)\to 1$ and arguing as in \cite[(7.41)--(7.42')]{selberg0},
one deduces that
\begin{equation}\label{gqtruncatednorm}
    \norm{E^{a}(z,1+ir)}^2=O_a\left(1+\left|\frac{\phi'}{\phi}(1+ir)\right|\right).
\end{equation}
For the Picard group we have 
\begin{equation*}
     \phi(s)=\frac{\pi\zeta_{\Q(i)}(s-1)}{(s-1)\zeta_{\Q(i)}(s)},
\end{equation*}
where $\zeta_{\Q(i)}$ is the Dedekind zeta function of the field $\Q(i)$.
Therefore, for some constant $\kappa>0$,
we have \cite[Lemma 9.2]{elstrodt}
\begin{equation}\label{gqscattering}
     \frac{\phi'}{\phi}(1+ir)=O\left(\log^\kappa (4+|r|)\right).
\end{equation}
Let $\mathcal{I}$ denote the integral in \eqref{gqcontinuous}.
Combining \eqref{gqtruncatedint}, \eqref{gqtruncatednorm} and \eqref{gqscattering},
multiplying by~$h$ and integrating over $r$, we deduce that
\[
\mathcal{I}\ll_f\int_{\R}|h_\pm(r)|\log^{\kappa}(4+|r|)\,dr.
\]
By Lemma \ref{gcl1} we can thus bound
\[
\mathcal{I}\ll_f Re^R \int_{\R} \frac{\log^\kappa(4+|r|)}{(1+|r|)^2}\,dr\ll_f R\,e^R,
\]
as claimed.
\end{proof}

In summary, combining the results of this section we deduce that
so far we proved the following.
\begin{corollary}\label{S4:cor}
Let $N^\pm(R,z)$ be the counting function defined in \eqref{gcmodifiedautokernel}.
Let $h_\pm(r)$ be as in \eqref{gc spectral} and let  $\epsilon,\eta\in(0,1)$, $R\gg 1$.
Let $f$ be a smooth compactly supported function on $\GmodHthree$.
Assume that Hypothesis \QVlink holds with $q\in[1,3]$.
Then, as $R\to\infty$, we have
\[
\int_{\GmodHthree} N^\pm(R,z)f(z)d\mu(z)
=
\frac{\pi}{2} \bar{f} e^{2R} + \bar{f}\sum_{r_j>0}h_\pm(r_j)
+
O_{f,\epsilon,q}\Bigl(e^{2R} \eta + Re^{R}\eta^{-\frac{q-1}{2}-\epsilon}\Bigr).
\]
\end{corollary}
\begin{proof}
Lemma \ref{gcl3} and Lemma \ref{S4:lemma:continuous} give the infinite
sum over $r_j>0$ and the last term in the error.
The leading term in Lemma \ref{gcl3} reads
\[
\bar{f}\, h_{\pm}(i)
\]
By definition of $h_\pm$ and by the evaluation at the point $i$
(see \eqref{gc spectral} and \eqref{specialvalues}), we get
\[
h_{\pm}(i) = h_{R\pm\eta}(i) = \pi(\sinh(2R\pm2\eta)-2(R\pm\eta)) = \frac{\pi}{2}e^{2R} + O(e^{2R}\eta + R),
\]
which gives the statement.
\end{proof}

\section{Proof of  Theorem \ref{ourmaintheorem}}\label{S5}

In the previous section we showed that we can expand the automorphic kernel
$N^\pm(R,z)$ in terms of the eigenvalues and the eigenfunctions,
and then showed how to remove the dependence on the eigenfunctions
by integrating against $f$ and by appealing to the quantum variance bound, see Corollary \ref{S4:cor}. 
It remains now to estimate a spectral exponential sum over the eigenvalues.

Let $T,X>1$ and consider the function $S(T,X)$ discussed in Section \ref{S1:STX}.
As we explained there, we have the estimate $S(T,X)\ll T^3$ uniformly in $X>1$.
Moreover, we assume that Hypothesis \STXlink holds with the exponents $(7/4+\theta,1/4)$,
for some $\theta\in[0,1/4]$ and in Section \ref{S1:STX} and Remark \ref{intro:rmk:variousbounds}
we discuss what values of $\theta$ are allowed.
Combining the two estimates, we obtain that
for $T,X>1$ and all $\epsilon\in(0,1)$ we have
\begin{equation}\label{expsumsestimate}
S(T,X) \ll_{\theta,\epsilon} \min\bigl(T^3,T^{\frac{7+4\theta}{4}+\epsilon}X^{\frac{1}{4}+\epsilon}+T^2\bigr).
\end{equation}
With these estimates at hand, we can bound the sum
involving the function $h_{\pm}$ that appears in Corollary \ref{S4:cor}.

\begin{lemma}\label{gcl4}
Let $\epsilon,\eta\in(0,1)$,
$R\gg 1$ and let $h_\pm$ be as in \eqref{gc spectral}. Then 
\begin{equation}\label{1402:eq010}
\sum_{r_j>0} h_\pm(r_j) \ll_{\theta,\epsilon} e^{R(\frac{6-4\theta}{5-4\theta}+\epsilon)}+e^{R(\frac{5}{4}+\epsilon)}\eta^{\frac{1}{4}-\theta-\epsilon}+e^{R}\eta^{-\epsilon}.
\end{equation}
\end{lemma}

\begin{proof}
By Lemma \ref{gcl1} and the fact that $r_j>1$ by \cite[Ch.~7, Proposition 6.2]{elstrodt},
we can write the sum as
    \begin{equation*}
        \sum_{r_j>1} A(R,r_j,\eta)e^{ir_j(R\pm\eta)} + B(R,r_j,\eta)e^{-ir_j(R\pm\eta)},
    \end{equation*}
    where the coefficients $A$ and $B$ satisfy the bounds in \eqref{0901:eq010}.
    The two parts above are treated similarly, so let us focus only on the sum with $A(R,r_j,\eta)$.
    We do a dyadic decomposition of the sum and look in windows of the type $r_j\in[T,2T)$. 
    Also, let $Y=e^{R\pm\eta}\asymp e^R$. Summation by parts yields
    \begin{equation*}
    \begin{aligned}
    \sum_{T\leq r_j<2T}A(R,r_j,\eta)Y^{ir_j}=A(R,2T,\eta)S(2T,Y)-A(R,T,\eta)S(T,Y)\\
    -\int_{T}^{2T}\partial_u A(R,u,\eta)S(u,Y)\,du.
    \end{aligned}
    \end{equation*}
    If $T\leq Y^{\frac{1}{5-4\theta}}$ we use the bounds in \eqref{0901:eq010},
    where it suffices to bound the minimum by~1,
    and the first bound from \eqref{expsumsestimate} obtaining
    \[
    \sum_{T\leq r_j<2T} A(R,r_j,\eta)e^{ir_j(R\pm\eta)} \ll e^R T \ll e^{R(\frac{6-4\theta}{5-4\theta})}.
    \]
    Since there are at most $O(\log Y)$ intervals with $T\leq Y^{\frac{1}{5-4\theta}}$, bounding each of them
    by the above estimate gives
    \begin{equation}\label{2608:eq001}
    \sum_{1<r_j\leq Y^{\frac{1}{5-4\theta}}} A(R,r_j,\eta)e^{ir_j(R\pm\eta)} \ll e^{R(\frac{6-4\theta}{5-4\theta}+\epsilon)}.
    \end{equation}
    When $T>Y^{\frac{1}{5-4\theta}}$ we use the bounds \eqref{0901:eq010} and the second bound in \eqref{expsumsestimate}, obtaining
    \begin{align*}
        A(R,T,\eta)S(T,Y) &\ll_{\theta,\epsilon} (T^{\theta-1/4+\epsilon}e^{R(5/4+\epsilon)}+e^R)\min(1,(\eta T)^{-2})
        \\
        \int_{T}^{2T}\partial_u A(R,u,n)S(u,Y)\,&du \ll_{\theta,\epsilon} (T^{\theta-1/4+\epsilon}e^{R(5/4+\epsilon)}+e^R)\min(1,(\eta T)^{-1}).
    \end{align*}
    It follows that
    \[
    \sum_{T\leq r_j<2T} A(R,r_j,\eta)e^{ir_j(R\pm\eta)} \ll_{\theta,\epsilon}
    \left(e^{R(5/4+\epsilon)} T^{\theta-1/4+\epsilon} +e^R\right) \min(1,(\eta T)^{-1}).
    \]
    Let $\sigma=(5-4\theta)^{-1}$.
	Summing over dyadic intervals with $T>Y^\sigma$ 
	and splitting cases for $(\eta T)^{-1}\leq1$ and $(\eta T)^{-1}>1$,
	we deduce that
	\begin{equation}\label{2608:eq002}
	\begin{split}
	\sum_{r_j>Y^\sigma} A(R,r_j,\eta)e^{ir_j(R\pm\eta)}
	&\ll_{\theta,\epsilon}
	e^{R(\frac{5}{4}+\epsilon)}
	\bigl(Y^{\sigma(\theta-1/4+\epsilon)}+
	\eta^{1/4-\theta-\epsilon}
	\bigr)
	+e^{R}\eta^{-\epsilon}
	\\
	&\ll_{\theta,\epsilon}
	e^{R(\frac{6-4\theta}{5-4\theta}+\epsilon)}
	+e^{R(\frac{5}{4}+\epsilon)}\eta^{1/4-\theta-\epsilon}
	+e^{R}\eta^{-\epsilon}.
	\end{split}
	\end{equation}
    Combining \eqref{2608:eq001} and \eqref{2608:eq002},
    the lemma follows immediately.
\end{proof}

\begin{proof}[Proof of Proposition \ref{smoothedkernel}]
Let $N^{\pm}(R,z)$ be the automorphic kernel defined in \eqref{gcmodifiedautokernel}. Using Corollary \ref{S4:cor} and Lemma \ref{gcl4} we obtain

\begin{equation*}
\begin{aligned}
&\int_{\GmodHthree} N^\pm(R,z)f(z)d\mu(z)
=
\frac{\pi}{2} \bar{f} e^{2R} \\ 
&+O_{f,\epsilon,\theta,q}\Bigl(e^{2R} \eta + Re^{R}\eta^{-\frac{q-1}{2}-\epsilon}+e^{R(\frac{6-4\theta}{5-4\theta}+\epsilon)}+e^{R(\frac{5}{4}+\epsilon)}\eta^{\frac{1}{4}-\theta-\epsilon}+e^{R}\eta^{-\epsilon}\Bigr).
\end{aligned}
\end{equation*}
Balancing the first and the second term yields the choice $\eta=e^{-\frac{2R}{q+1}}$.
For this choice of $\eta$ the first, second and fifth term on the above error term contribute $e^{(\frac{2q}{q+1}+\epsilon)R}$ and for any $\theta\in[0,1/4]$ and any $q\in[1,3]$ the fourth term is dominated by the third, with the equality given at $\theta=1/4$ independently of $q$. With all the above the proof of  Proposition \ref{smoothedkernel} is now complete.
\end{proof}


\end{document}